\documentclass[hidelinks,onefignum,onetabnum]{siamart220329}

\usepackage{lipsum}
\usepackage{amsfonts}
\usepackage{graphicx}
\usepackage{epstopdf}
\usepackage{algorithmic,nicematrix,colortbl,nicematrix}
\usepackage{amsmath, amssymb, enumitem,xcolor, xfrac, amsfonts,tikz}
\usepackage{float}
\usepackage{nicematrix}

\ifpdf
  \DeclareGraphicsExtensions{.eps,.pdf,.png,.jpg}
\else
  \DeclareGraphicsExtensions{.eps}
\fi

\newcommand{\mbf}[1]{\mathbf{#1}}
\newcommand{\mbb}[1]{\mathbb{#1}}

\newcommand{\mcf}[1]{\mathcal{#1}}

\newenvironment{cpf}{\begin{trivlist} \item[] {\em Proof of Claim.}}{\hspace*{\stretch{1}} $\diamond$ \end{trivlist}}

\DeclareMathOperator{\cl}{cl}
\DeclareMathOperator{\DG}{DG}
\DeclareMathOperator{\elem}{\varepsilon}
\newcommand{\rank}{\operatorname{rank}}

\newcommand{\si}{\operatorname{si}}
\newcommand{\GF}{\operatorname{GF}}

\newsiamremark{remark}{Remark}
\newsiamthm{claim}{Claim}
\newsiamthm{conjecture}{Conjecture}
\newsiamthm{problem}{Problem}

\headers{The column number and exclusions in $\Delta$-modular matrices}{J. Paat, I. Stallknecht, Z. Walsh, and L. Xu}

\title{On the column number and forbidden submatrices for $\Delta$-modular matrices}

\author{Joseph Paat\thanks{Sauder School of Business, University of British Columbia, Canada}
\and Ingo Stallknecht\thanks{Department of Mathematics, ETH Z\"{u}rich, Switzerland}
\and Zach Walsh\thanks{School of Mathematics, Georgia Institute of Technology, USA}
\and Luze Xu\thanks{Department of Mathematics, University of California Davis, USA}}

\usepackage{amsopn}

\begin{document}

\maketitle

\begin{abstract}
%
%
%
%
%
%
%
An integer matrix $\mbf{A}$ is $\Delta$-modular if the determinant of each $\rank(\mbf{A}) \times \rank(\mbf{A})$ submatrix of $\mbf{A}$ has absolute value at most $\Delta$.
The study of $\Delta$-modular matrices appears in the theory of integer programming, where an open conjecture is whether integer programs defined by $\Delta$-modular constraint matrices can be solved in polynomial time if $\Delta$ is considered constant. 
The conjecture is only known to hold true when $\Delta \in \{1,2\}$. 
In light of this conjecture, a natural question is to understand structural properties of $\Delta$-modular matrices. 
We consider the column number question --- how many nonzero, pairwise non-parallel columns can a rank-$r$ $\Delta$-modular matrix have?
We prove that for each positive integer $\Delta$ and sufficiently large integer $r$, every rank-$r$ $\Delta$-modular matrix has at most $\binom{r+1}{2} + 80\Delta^7 \cdot r$ nonzero, pairwise non-parallel columns, which is tight up to the term $80\Delta^7$.
This is the first upper bound of the form $\binom{r+1}{2} + f(\Delta)\cdot r$ with $f$ a polynomial function.
Underlying our results is a partial list of matrices that cannot exist in a $\Delta$-modular matrix. 
We believe this partial list may be of independent interest in future studies of $\Delta$-modular matrices.
\end{abstract}

\section{Introduction}\label{secIntro}
%
Given a matrix $\mbf{A}\in \mbb{Z}^{m\times n}$ and a vector $\mbf{b}\in \mbb{Z}^m$, the integer programming problem is to find a vector in $\{\mbf{x}\in \mbb{Z}^n: \mbf{A}\mbf{x}\le \mbf{b}\}$ or declare the set is empty. 
The integer programming problem is $\mathcal{NP}$-Hard in general~\cite{Kar21}.
In light of this, an interesting question becomes identifying parameters of $\mbf{A}$ and $\mbf{b}$ for which the problem can be solved in polynomial time. 
One such parameter is $n$.
Lenstra proves in~\cite{Lenstra1983} that the problem can be solved in polynomial time if $n$ is fixed.
Another parameter of interest is the largest determinant in the constraint matrix.


Given a positive integer $\Delta$, the matrix $\mbf{A}$ is {\bf $\Delta$-modular} if the determinant of each $\rank(\mbf{A})\times \rank(\mbf{A})$ submatrix of $\mbf{A}$ has absolute value at most $\Delta$. 
A $1$-modular matrix is {\bf unimodular}.
If $\mbf{A}$ is unimodular, then the integer programming problem can be solved in strongly polynomial time because every vertex of the polyhedron $\{\mbf{x}\in \mbb{R}^n: \mbf{A}\mbf{x}\le\mbf{b}\}$ is integral and thus it suffices to solve the corresponding linear relaxation, which can be solved in strongly polynomial time~\cite{Tar1986}; this result also holds for the optimization version where one maximizes a linear function over the integer points in the polyhedron.
Artmann, Weismantel, and Zenklusen show that if $\mbf{A}$ is $2$-modular, or {\bf bimodular}, then the integer programming problem (and the optimization version) can be solved in strongly polynomial time~\cite{AWZ2017}.
There are other results looking beyond bimodular matrices, see, e.g.,~\cite{AEGOVW2016,FiJoWeYu2022,abc2021,NaSaZe2022,PSW2020,VC2009}, but a polynomial time algorithm for the integer programming problem for fixed $\Delta \ge 3$ remains an open problem.
A natural line of thought to explore this open problem is to investigate the structural properties of $\Delta$-modular matrices.
Perhaps one of the most basic properties is the maximum size of a $\Delta$-modular matrix of fixed rank.
We say two columns $\mbf{x}$ and $\mbf{y}$ of a matrix $\mbf{A}$ are {\bf non-parallel} if they are linearly independent.

\begin{problem} \label{extremal}
For positive integers $\Delta$ and $r$, what is the maximum number of nonzero, pairwise non-parallel columns of a rank-$r$ $\Delta$-modular matrix?
\end{problem}

The {\bf column number question} posed in Problem~\ref{extremal} has applications in the theory of integer programming, 
see, e.g.,~\cite{LPSX2021} for a discussion on the distance between vertices of polyhedra and integer points inside, and see~\cite{AEGOVW2016,PSW2020}~\cite[\S 3]{DHK2012} for algorithms whose analysis relies on the column number. 

A classic result of Heller~\cite{H1957} shows that a rank-$r$ unimodular matrix has at most $\binom{r + 1}{2}$ nonzero, pairwise non-parallel columns.
More recently, Lee, Paat, Stallknecht, and Xu \cite{LPSX2021} prove that a rank-$r$ bimodular matrix has at most $\binom{r+1}{2} + r - 1$ nonzero, pairwise non-parallel columns, except for the case $r=3,5$, where $\binom{r+1}{2} + r$ is the answer to Problem~\ref{extremal}.
Oxley and Walsh \cite{OW2021} independently demonstrate that $\binom{r+1}{2} + r - 1$ is an upper bound on the number of elements for every matroid in a family that includes bimodular matrices as a special case, but their result only holds for $r$ sufficiently large.
The results of Heller, Lee et al., and Oxley and Walsh are tight in the sense that they provide an answer to Problem \ref{extremal} for $\Delta \in \{1,2\}$.

Problem \ref{extremal} is open for $\Delta \ge 3$.
Many approximate answers have been given; see, e.g.,~\cite{A1990,GWZ2018,KunMat1990,K1990,Lee1989}.
We highlight a few recent contributions.
Geelen, Nelson, and Walsh \cite{GNW2021} proved an upper bound of $\binom{r+1}{2} + c\cdot r$ for a constant $c$.
However, the constant $c$ is at least a doubly exponential function of $\Delta$. 
Lee, Paat, Stallknecht, and Xu \cite{LPSX2021} provide the following result (Theorem~\ref{poly_Delta_bound}), which is the first upper bound that is polynomial in both $r$ and $\Delta$; we make use of this result in our main proof.
%
Averkov and Schymura \cite{AS2022} prove another polynomial upper bound of $O(\Delta r^4)$.

\begin{theorem}[Lee, Paat, Stallknecht, and Xu \cite{LPSX2021}] \label{poly_Delta_bound}
Let $r$ and $\Delta$ be positive integers and $\mbf{A}$ be a rank-$r$ $\Delta$-modular matrix.
Then $\mbf{A}$ has at most $\Delta^2\binom{r+1}{2}$ nonzero, pairwise non-parallel columns.
\end{theorem}

%
Geelen et al.'s result demonstrates the smallest possible coefficient in front of the quadratic term $\binom{r+1}{2}$; the fact that the quadratic dependence on $r$ is best possible follows from Heller's work.
However, Geelen et al.'s result has a large dependence on $\Delta$.
In contrast, Averkov and Schymura's result has a best-possible dependence on $\Delta$ (see their work or Lee et al.'s for a lower bound), albeit at the expense of a larger dependence on $r$. 
Lee et al.'s result provides a middle ground, but there still exists a {\it multiplicative} dependence between $\Delta^2$ and $r^2$.
Our main result demonstrates that for sufficiently large $r$, the column number bound differs from (the best possible) quadratic dependence on $r$ by an {\it additive} factor of the form $f(\Delta) \cdot r$ with $f$ a polynomial.
This is the first polynomial bound on the column number question that provides this additive dependence on $f(\Delta)$.
Our result greatly improves on Geelen et al.

\begin{theorem} \label{main}
For each positive integer $\Delta$ and each sufficiently large integer $r$, if $\mbf{A}$ is a rank-$r$ $\Delta$-modular matrix, then $\mbf{A}$ has at most $\binom{r + 1}{2} + 80 \Delta^7 \cdot r$ nonzero, pairwise non-parallel columns.
\end{theorem}

Our analysis of Problem~\ref{extremal} considers the set $\mcf{M}_{\Delta}$ of {\bf $\Delta$-modular matroids}, which are the matroids with a representation over $\mbb{R}$ as a $\Delta$-modular matrix.
See the end of Section~\ref{secIntro} for basic matroid definitions.
The set $\mcf{M}_1$ of {\bf regular} matroids is precisely the set of matroids with a representation over every field, as Tutte proves in~\cite{Tutte}.
The set $\mcf{M}_2$ of bimodular matroids is studied by Oxley and Walsh~\cite{OW2021} in their analysis of Problem~\ref{extremal}.
The set $\mcf{M}_{\Delta}$ is considered by Geelen et al. in their analysis of Problem~\ref{extremal}; among other results, they prove that $\mcf{M}_{\Delta}$ is minor-closed and every matroid in $\mcf{M}_{\Delta}$ is representable over every field with characteristic greater than $\Delta$; see~\cite[\S 8.6]{GNW2021}.
Although Problem~\ref{extremal} is stated for matrices, one could instead phrase the question for matroids: what is the maximum number of elements of a simple rank-$r$ matroid in $\mathcal M_{\Delta}$?
By considering matroids as opposed to matrices, we will be able to use technology derived from the matroid literature, e.g., results in~\cite{GNW2021}.
%

Underlying our proof of Theorem~\ref{main} is our second main result, which is a collection of {\it forbidden matroids} for $\mcf{M}_{\Delta}$; see Propositions~\ref{lemspikes},~\ref{lemstacks} and~\ref{single_element_extension_of_clique}.
These forbidden matroids correspond to forbidden matrices for the class of $\Delta$-modular matrices.
There has been previous work on forbidden matroids in $\mcf{M}_{\Delta}$, particularly for small values of $\Delta$.
For instance, the set $\mcf{M}_1$ is completely characterized by its three excluded minors (by excluded, we refer to minor-minimal forbidden matroids): the uniform matroid $U_{2,4}$, the Fano plane $F_7$, and its dual $F_7^*$.
Oxley and Walsh~\cite{OW2021} find several excluded minors for $\mcf{M}_2$, but remark that it is not a complete list. 
For general $\Delta$, Geelen et al. \cite[Proposition 8.6.1]{GNW2021} show that the uniform matroid $U_{2,2\Delta + 2}$ is not in $\mcf{M}_{\Delta}$, and that the signed-graphic Dowling geometry $\DG(2\lfloor\log_2 \Delta\rfloor + 2, \{1,-1\})$ is not in $\mcf{M}_{\Delta}$.
%
Our proof of Theorem~\ref{main} uses similar techniques to those of Geelen et al., but we improve their bound by identifying three types of matroids whose behavior is limited within the set of $\Delta$-modular matroids: stacks, spikes, and single-element extensions of cliques.
In Section~\ref{secExcludedMatroids} we define these three types of matroids and demonstrate the extent to which they can occur in $\mcf{M}_{\Delta}$.
As a final note, we mention that we are identifying matrices that are forbidden in any $\Delta$-modular matrix, not just {\it maximum} $\Delta$-modular matrices, i.e., matrices that achieve the maximum number of columns in Problem~\ref{extremal}.
We believe these forbidden matrices may be of independent interest in future study of $\Delta$-modular matrices, not necessarily maximum ones.
In Section~\ref{secSCC} we prove Theorem \ref{main} for matroids in $\mcf{M}_{\Delta}$ with a spanning clique restriction, and then in Section~\ref{secMainproof} we prove the full result.
Finally, we discuss several related open problems in Section~\ref{secFutureWork}.

\medskip

\noindent{\bf Notation and definitions.}
We use lower case bold font for vectors, e.g., $\mbf{x}\in \mbb{R}^d$, and upper case bold font for matrices, e.g., $\mbf{A} \in \mbb{Z}^{m\times n}$.
For a vector $\mbf{x} \in \mbb{R}^d$, we let $x_i$ denote the $i$th component of $\mbf{x}$.
We use $\mbf{e}_1, \dotsc, \mbf{e}_d \in \mbb{R}^d$ to denote the standard unit vectors in $\mbb{R}^d$.
We refer to column vectors inline; given $\mbf{x} \in \mbb{R}^{d_1}$ and $\mbf{y} \in \mbb{R}^{d_2}$ we use the notation $(\mbf{x}, \mbf{y})$ to denote the column vector
\[
\left[\begin{array}{c}\mbf{x}\\\mbf{y}\end{array}\right] \in \mbb{R}^{d_1+d_2}.
\]
We use $\mbf{1}_{d}$ and $\mbf{0}_{d}$ to denote the $d$-dimensional vector of all ones and all zeroes, respectively, and $\mbf{I}_d$ to denote the $d\times d$ identity matrix; we omit $d$ if the dimension is clear from context.

Matroids were introduced independently by Nakasawa \cite{Nakasawa1, Nakasawa2, Nakasawa3} and Whitney \cite{Whitney1935} as a combinatorial abstraction of rank functions that arise in linear algebra.
We use standard notation in matroid theory, and for a detailed introduction we direct the reader to \cite{Oxley}.

A {\bf matroid} is a pair $M = (E,r)$, where $E = E(M)$ is a finite ground set and $r:2^E \to \mathbb{Z}$ is a rank function that is submodular and non-decreasing.
We write matroids, their ground sets, and elements within using italic font.
Given $t \in E$, we denote the {\bf deletion} matroid of $M$ by $t$ as $M\backslash t$ and the {\bf contraction} matroid of $M$ by $t$ as $M / t$. 
A {\bf circuit} $X \subseteq E(M)$ is an inclusion-wise minimal set satisfying $r(X - b) = r(X)$ for all $b \in X$.
A {\bf loop} is a circuit of size $1$, and two elements are {\bf parallel} if they form a circuit of size $2$.
A matroid is {\bf simple} if it contains no loops or parallel elements, and the {\bf simplification} $\si(M)$ of $M$ is constructed by selecting one element from each parallel class.
Given a set $X \subseteq E(M)$, the {\bf restriction} $M|X$ is the matroid defined by restricting $E(M)$ to $X$ and $r$ to $2^X$.
We say that a set $X \subseteq E$ {\bf spans} an element $t \in E$ if $X$ and $X \cup \{t\}$ have the same rank in $M$; the {\bf closure} of $X$ in $M$, denoted by $\cl_M(X)$, consists of all elements of $E(M)$ spanned by $X$.
Given two matroids $M$ and $M'$, we write $M \cong M'$ if the two matroids are {\bf isomorphic}.

A {\bf point} of a matroid $M$ is a maximal rank-$1$ subset of elements. 
We write $|M|$ and $\epsilon(M)$ for the number of elements and the number of points in $M$, respectively.
A {\bf line} of a matroid $M$ is a maximal rank-$2$ subset of elements, and a line is {\bf long} if it contains at least three points.
We refer to $M(K_r)$, the graphic matroid of the complete graph on $r$ vertices, as a {\bf clique}.
The canonical representation of $M(K_r)$ over any field is $[\begin{array}{@{\hskip .1 cm}c@{\hskip .15 cm}c@{\hskip .1 cm}}\mbf{I}_{r-1}&\mbf{D}_{r-1}\end{array}] \in \mbb{Z}^{(r-1)\times ((r-1)+\binom{r-1}{2})}$, where the columns of $\mbf{D}_{r-1}$ are $\mbf{e}_i - \mbf{e}_j$ for all $1 \le i < j \le r-1$.
If $\mbf{A}$ is a matrix over a field $\mathbb{F}$ with columns indexed by $E$, then the function $r : 2^E \to \mathbb Z$ defined by $r(X) := \rank(\mbf{A}[X])$ defines a matroid on $E$, where $\mbf{A}[X]$ is the set of columns of $\mbf{A}$ corresponding to $X$.
The corresponding matroid is {\bf $\mathbb F$-representable}.

\section{Forbidden matroids for \texorpdfstring{$\mcf{M}_{\Delta}$}{}}\label{secExcludedMatroids}

In this section we identify several natural classes of matroids that are not in $\mcf{M}_{\Delta}$.
Instead of finding specific forbidden minors for $\mcf{M}_{\Delta}$, we identify entire families of matroids that are not in $\mcf{M}_{\Delta}$.
We choose to focus on families of $\mathbb R$-representable matroids, as these shed more light on properties of $\mcf{M}_{\Delta}$.
%
For example, the Fano plane $F_7$ is an excluded minor for $\mcf{M}_{\Delta}$ for all $\Delta \ge 1$, but this is not helpful for understanding which integer matrices are $\Delta$-modular because $F_7$ cannot be represented as an integer matrix.
We focus on three families of matroids: spikes, stacks, and single-element extensions of cliques.

Stacks, spikes, and extensions of cliques are all of fundamental significance in extremal matroid theory.
Geelen and Nelson prove in~\cite{GN} that if $\mcf{M}$ is a minor-closed class $\mcf{M}_{\Delta}$, then one can always find an extremal matroid containing a spanning clique restriction, or has a structure that naturally leads to a large collection of spikes or stacks.
Our proof of Theorem~\ref{main} uses Theorem \ref{reduction}, which is a refinement of~\cite{GN}, and its use is facilitated by the exclusion of certain stacks, spikes, and extensions of cliques from $\mcf{M}_{\Delta}$.

\subsection{Spikes}
A {\bf spike} is a simple matroid $S$ with an element $t$, called the {\bf tip}, such that the simplification of $S/t$ is a circuit and each parallel class of $S/t$ has size two.
The following matrix is an example of an $\mathbb R$-representation of a rank-$3$ spike, where the first column corresponds to the tip:
\[
\left[ 
\begin{array}{ccccccc}
1 & 1 & 1 & 1 & 0 & 0 & 0 \\
0 & 1 & 0 & 1 & 1 & 0 & 1 \\
0 & 0 & 1 & 1 & 0 & 1 & 1
\end{array}
\right].
\]
%
%

%
The following result is our first forbidden matrix result.
We show that $\mcf{M}_{\Delta}$ only contains spikes of bounded rank.
%
%

\begin{proposition} \label{lemspikes}
For each positive integer $\Delta$, the set $\mcf{M}_{\Delta}$ contains no spike of rank greater than $2\Delta$.
\end{proposition}
\begin{proof}
Suppose $\mcf{M}_{\Delta}$ contains a rank-$r$ spike $S$ with tip $t$.
Without loss of generality, we can assume that $t=(1,\mathbf{0}_{r-1})$, $S$ is $\mathbb{R}$-represented by a $\Delta$-modular matrix $\mbf{A}\in\mathbb{Z}^{r\times (2r+1)}$,
and $\si(S/t)$ is $\mathbb{R}$-represented by the matrix
\[
\mbf{C} := \left[ 
\begin{array}{@{\hskip 0.1 cm}c@{\hskip 0.15 cm}c@{\hskip 0.15 cm}c@{\hskip 0.1 cm}}
\mbf{b}^1 & \cdots & \mbf{b}^{r}
\end{array}
\right]
\]
in $\mbb{Z}^{(r-1)\times r}$.
Given that $\mbf{C}$ is a circuit, we can write $\mbf{b}^r = \sum_{i=1}^{r-1} \lambda_i \mbf{b}^i$ for nonzero coefficients $\lambda_1, \dotsc, \lambda_{r-1}$.
Geelen et al.~\cite[Proposition 8.6.1]{GNW2021} show that $\mcf{M}_{\Delta}$ is minor-closed, so $\mbf{C}$ inherits the $\Delta$-modular property from $\mbf{A}$.
By Cramer's Rule, we have $\lambda_1, \dotsc, \lambda_{r-1} \in \sfrac{1}{\delta} \cdot \mbb{Z}$ for 
\[
\delta :=|\det[\begin{array}{@{\hskip .1cm}c@{\hskip .15 cm}c@{\hskip.15cm}c@{\hskip .1 cm}}\mbf{b}^1 &\cdots&\mbf{b}^{r-1}\end{array}]|
\]
which lies in $\in \{1,\dotsc,\Delta\}$.

For $\mbf{x} \in \mbb{Z}^r$, define the matrix
\[
\mbf{B}(\mbf{x}) := 
\left[
\begin{array}{@{\hskip 0.1 cm}c@{\hskip .15 cm}c@{\hskip .15 cm}c@{\hskip .15 cm}c@{\hskip 0.1 cm}}
x_1 & \cdots & x_{r-1} & x_r \\[.1 cm]
\mbf{b}^1 & \cdots & \mbf{b}^{r-1} & \mbf{b}^r
\end{array} 
\right] \in \mbb{Z}^{r \times r}.
\]
The matrix $\mbf{B}(\mbf{x})$ is equivalent up to elementary column operations to 
\[
\left[
\begin{array}{@{\hskip 0.1 cm}c@{\hskip .15 cm}c@{\hskip .15 cm}c@{\hskip .15 cm}c@{\hskip 0.1 cm}}
x_1 & \cdots & x_{r-1} & x_r - \sum_{i=1}^{r-1} \lambda_i x_i \\[.1 cm]
\mbf{b}^1 & \cdots & \mbf{b}^{r-1} & \mbf{0}
\end{array} 
\right].
\]
Therefore, if $\mbf{x}$ is chosen such that $\mbf{B}(\mbf{x})$ is a submatrix of $\mbf{A}$, then 
\[
\left|\det\mbf{B}(\mbf{x})\right| = \delta\ \bigg|x_r - \sum_{i=1}^{r-1} \lambda_i x_i\bigg|\le \Delta.
\]

Consider the set
\[
\Phi := \left\{x_r - \sum_{i=1}^{r-1} \lambda_i x_i :\ \mbf{B}(\mbf{x})~\text{is a submatrix of}~\mbf{A}\right\}
\]
The numbers $\lambda_1,\dotsc, \lambda_{r-1}\in \sfrac{1}{\delta}\cdot \mbb{Z}$ are nonzero.
Moreover, for each $i \in \{1,\dotsc,r\}$ there exists at least two integers $\widehat{x}_i < \overline{x}_i$ such that $(\widehat{x}_i, \mbf{b}^i),(\overline{x}_i, \mbf{b}^i) \in \mbf{A}$ because $S$ is a spike.
Therefore, 
\begin{align*}
\max \Phi  - \min \Phi &\ge \left(\overline{x}_r - \sum_{\substack{i=1,\\ \lambda_i > 0}}^{r-1} \lambda_i \widehat{x}_i - \sum_{\substack{i=1,\\ \lambda_i < 0}}^{r-1} \lambda_i \overline{x}_i\right) - \left(\widehat{x}_r - \sum_{\substack{i=1,\\ \lambda_i > 0}}^{r-1} \lambda_i \overline{x}_i- \sum_{\substack{i=1,\\ \lambda_i < 0}}^{r-1} \lambda_i \widehat{x}_i\right)\\
& = (\overline{x}_r -\widehat{x}_r) + \sum_{\substack{i=1,\\ \lambda_i > 0}}^{r-1} \lambda_i  (\overline{x}_i-\widehat{x}_i) - \sum_{\substack{i=1,\\ \lambda_i < 0}}^{r-1} \lambda_i (\overline{x}_i-\widehat{x}_i)\\
&\ge 1 + \sum_{i=1}^{r-1} |\lambda_i|\\
&\ge 1+\frac{r-1}{\delta}.
\end{align*}
%
%
The $\Delta$-modular property of $\mbf{A}$ implies that $\delta|\max \Phi|\le \Delta$ and $\delta|\min \Phi|\le \Delta$.
Therefore, $2\Delta\ge \delta(\max \Phi - \min \Phi)\ge \delta + r-1$.
Hence, $r\le 2\Delta+1-\delta \le 2\Delta$.
\end{proof}

The bound in Proposition~\ref{lemspikes} is tight for all $\Delta$ due to the following matrix: 
\[
\left[
\begin{array}{@{\hskip .1 cm}c@{\hskip .15 cm}c@{\hskip .15 cm}c@{\hskip .15 cm}c@{\hskip .15 cm}c@{\hskip .1 cm}}
1 & \mathbf{0}^\top_{2\Delta-1} & \mathbf{1}^\top_{2\Delta-1} & \Delta-1 & \Delta\\
\mathbf{0}_{2\Delta-1} & \mbf{I}_{2\Delta-1} & \mbf{I}_{2\Delta-1} & \mathbf{1}_{2\Delta-1} & \mathbf{1}_{2\Delta-1}
\end{array}
\right].
\]
Every spike of rank greater than $2\Delta + 1$ has a rank-$(2\Delta + 1)$ spike-minor, so we need only use the fact that $\mcf{M}_{\Delta}$ contains no spike of rank exactly $2\Delta + 1$, which is a finite collection of matroids.
The matrix above shows that there are $\mathbb R$-representable rank-$n$ spikes for each integer $n \ge 3$; however, not all spikes are $\mathbb R$-representable~\cite[Proposition 12.2.20]{Oxley}.

\subsection{Stacks}
The second forbidden type of matroid is called a {\bf stack}.
Roughly speaking, a stack is a collection of bounded-size restrictions, each of which is not a member of some fixed class $\mathcal O$.
More precisely, for integers $m \ge 2$ and $h \ge 1$ and a collection $\mathcal O$ of matroids, a matroid $M$ is an {\bf$(\mcf{O},m,h)$-stack} if there are disjoint sets $P_1, P_2,\dots,P_h \subseteq E(M)$ such that

\smallskip
\begin{itemize}[leftmargin = *]
    \item $\cup_{i=1}^h P_i$ spans $M$, and
    \item for each $i \in \{1, \dotsc, h\}$, the matroid $(M/(P_1 \cup \cdots \cup P_{i-1}))|P_i$ has rank at most $m$ and is not in $\mcf{O}$.
\end{itemize}
\smallskip

Stacks were used in~\cite{DensestPG,GNW2021,DensePG} to find extremal functions for minor-closed sets of matroids.
Our definition generalizes the original definition from \cite{HalesJewett}.
We will always take $\mcf{O}$ to be the set $\mcf{M}_1$ of regular matroids.
For example, the direct sum of $h$ matroids, each of which is $U_{2,4}$ or the Fano plane $F_7$, is an $(\mathcal M_1, 3, h)$-stack, because $U_{2,4}$ and $F_7$ are not in $\mathcal M_1$ and each has rank at most three.

The following proposition, which is our second forbidden matrix result, generalizes~\cite[Claim 8.6.1.4]{GNW2021}, where it was shown that a signed-graphic Dowling geometry of sufficiently large rank is not $\Delta$-modular because it contains a stack of copies of $U_{2,4}$.
Note that if a matroid contains an $(\mathcal M_1, m, h)$-stack as a restriction, then it also contains an $(\mathcal M_1, m, h')$-stack as a restriction for all $h'\le h$.

\begin{proposition} \label{lemstacks}
For all positive integers $\Delta$ and $m$ with $m\ge 2$, there is no $(\mcf{M}_1, m, \lfloor\log_2\Delta\rfloor + 1)$-stack in the set $\mcf{M}_{\Delta}$.
\end{proposition}
\begin{proof}
For the sake of contradiction, assume that $\mcf{M}_{\Delta}$ contains such a stack $M$. 
There are disjoint sets $P_1, P_2,\dotsc,P_h \subseteq E(M)$ with $h=\lfloor\log_2\Delta\rfloor+1$ such that
%

\smallskip
\begin{enumerate}[label=(\alph*), leftmargin = *]
    \item\label{stackProp1} $\cup_{i=1}^h P_i$ spans $M$, and
    \item\label{stackProp2} for each $i \in \{1, \dotsc, h\}$, the matroid $(M/(P_1 \cup \cdots \cup P_{i-1}))|P_i$ has rank at most $m$ and is not in $\mcf{M}_1$.
\end{enumerate}
\smallskip

Without loss of generality, we can assume that $M$ is $\mathbb{R}$-representable by a $\Delta$-modular matrix $\mbf{A}$ of the upper block-triangular form
%
\[
\begin{bmatrix}
\mbf{A}_1 & * & \cdots & *\\
 & \mbf{A}_2 & \cdots & *\\
 &  & \ddots & \vdots\\
 & & & \mbf{A}_h
\end{bmatrix}.
\]
For each $i\in\{1, \dotsc, h\}$, the matroid $(M/(P_1 \cup \cdots \cup P_{i-1}))|P_i$ is $\mathbb{R}$-representable by $\mbf{A}_i$.
Furthermore, Property~\ref{stackProp2} states that $\mbf{A}_i$ is not unimodular. 
Hence, there exists a submatrix $\mbf{B}_i$ of $\mbf{A}_i$ with $|\det \mbf{B}_i|\ge 2$.
The block-diagonal matrix with blocks $\mbf{B}_1,\mbf{B}_2,\dotsc,\mbf{B}_h$ is a submatrix of $\mbf{A}$, and its absolute determinant is $\prod_{i=1}^h|\det \mbf{B}_i|\ge 2^h>\Delta$.
However, this contradicts that $M \in \mcf{M}_{\Delta}$.
\end{proof}

In other words, Proposition~\ref{lemstacks} states that we cannot sequentially contract more than $\lfloor\log_2\Delta\rfloor$ non-regular restrictions in a $\Delta$-modular matroid.
%
Certainly not every stack is $\mathbb R$-representable, but many are.
For example, any stack of uniform matroids is $\mbb{R}$-representable.
%
We comment that the bound of the previous proposition is tight, for example, for the direct sum of $\lfloor\log_2\Delta\rfloor$ copies of $U_{2,4}$.

\subsection{Single-Element Extensions of a Clique}
In this subsection we consider single-element extensions of a clique.
The following result is our third forbidden matrix result.

\begin{proposition}\label{single_element_extension_of_clique}
Let $\Delta$ be a positive integer.
If a rank-$r$ matroid $M \in \mcf{M}_{\Delta}$ has a set $X \subseteq E(M)$ such that $M|X \cong M(K_{r+1})$, then every element of $M$ is spanned by a set of at most $\Delta$ elements of $X$.
\end{proposition}

\begin{proof}
Suppose $M$ is $\mathbb{R}$-represented by a $\Delta$-modular matrix $\mbf{A}$, and the clique $M|X$ is $\mathbb{R}$-represented by a submatrix $\mbf{B}$.
By~\cite[Proposition 6.6.5]{Oxley}, the matrix $\mbf{B}$ is row equivalent to $[\begin{array}{@{\hskip .1 cm}c@{\hskip .15 cm}c@{\hskip .1 cm}}\mbf{I}_r&\mbf{D}_r\end{array}]$, where $\mbf{D}_r$ is the network matrix with exactly one 1 and -1. 
We assume that $\mbf{B} = \mbf{Q}\ [\begin{array}{@{\hskip .1 cm}c@{\hskip .15 cm}c@{\hskip .1 cm}}\mbf{I}_r&\mbf{D}_r\end{array}]$, where $\mbf{Q}\in\mathbb{Z}^{r\times r}$ and $1\le |\det \mbf{Q}|\le\Delta$.

Set $q := |\det \mbf{Q}|$.
Let $\mbf{a}$ be a column in $\mbf{A}$ but not in $\mbf{B}$, and consider the matrix $[\begin{array}{@{\hskip .1 cm}c@{\hskip .15 cm}c@{\hskip .15 cm}c@{\hskip .1 cm}}\mbf{I}_r&\mbf{D}_r&q\mbf{Q}^{-1}\mbf{a}\end{array}] \in \mbb{Z}^{r \times(r+\binom{r}{2}+1)}$.
We have $q\mbf{Q}^{-1}\mbf{a}\in\mathbb{Z}^r$ by Cramer's Rule because $q=|\det \mbf{Q}|$.
Furthermore, for any $r\times(r-1)$ submatrix $\mbf{C}$ of $[\begin{array}{@{\hskip .1 cm}c@{\hskip .15 cm}c@{\hskip .1 cm}}\mbf{I}_r & \mbf{D}_r\end{array}]$, we have
\[
\left|\det[\begin{array}{@{\hskip .1 cm}c@{\hskip .15 cm}c@{\hskip .1 cm}}\mbf{C} & q\mbf{Q}^{-1}\mbf{a}\end{array}]\right|=q\left|\det[\begin{array}{@{\hskip .1 cm}c@{\hskip .15 cm}c@{\hskip .1 cm}}\mbf{C}&\mbf{Q}^{-1}\mbf{a}\end{array}]\right|=\left|\det[\begin{array}{@{\hskip .1 cm}c@{\hskip .15 cm}c@{\hskip .1 cm}}\mbf{Q}\mbf{C}&\mbf{a}\end{array}]\right|\le \Delta
\]
because $[\begin{array}{@{\hskip .1 cm}c@{\hskip .15 cm}c@{\hskip .1 cm}}\mbf{Q}\mbf{C}&\mbf{a}\end{array}]$ is a submatrix of $\mbf{A}$.
Therefore, $[\begin{array}{@{\hskip .1 cm}c@{\hskip .15 cm}c@{\hskip .15 cm}c@{\hskip .1 cm}}\mbf{I}_r&\mbf{D}_r&q\mbf{Q}^{-1}\mbf{a}\end{array}]$ is $\Delta$-modular. 

Set $\mbf{f} := q\mbf{Q}^{-1}\mbf{a}$.
In order to prove that $\mbf{a}$ is spanned by a set of at most $\Delta$ elements in $\mbf{B}$, it suffices to show that $\mbf{f}$ is spanned by a set of most $\Delta$ columns in $[\begin{array}{@{\hskip .1 cm}c@{\hskip .15 cm}c@{\hskip .1 cm}}\mbf{I}_r&\mbf{D}_r\end{array}]$.
For a set $\{i_1, \dotsc, i_s\} \subseteq \{1, \dotsc, r\}$, consider the matrix
\[
\mbf{G}(\{i_1, \dotsc, i_s\}) := \left[
\begin{array}{rrcrc}
1 & 1 & \cdots & 1 & f_{i_1}\\
-1 & 0 & \cdots & 0 & f_{i_2}\\
0 & -1 & \cdots & 0& f_{i_3}\\
\vdots & \vdots & \cdots & \vdots & \vdots \\
0 & 0 & \cdots & -1 & f_{i_s}
\end{array}
\right],
\]
which is a submatrix of $[\begin{array}{@{\hskip .1 cm}c@{\hskip .15 cm}c@{\hskip .15 cm}c@{\hskip .1 cm}}\mbf{I}_r&\mbf{D}_r&\mbf{f} \end{array}]$.
We have $|\det \mbf{G}(\{i_1, \dotsc, i_s\})| = |\sum_{j=1}^s f_{i_j}|$.
By choosing $\{i_1, \dotsc, i_s\}$ to index the positive (respectively, the negative) entries of $\mbf{f}$, we see that the positive entries (respectively, the negative entries) of $\mbf{f}$ sum to at most $\Delta$ (respectively, at least $-\Delta$).

\begin{claim}\label{claimExtensionsClique}
Let $k$ be a positive integer such that the positive entries of $\mbf{f}$ sum to at most $k$ and the absolute value of the negative entries of $\mbf{f}$ sum to at most $k$.
Then $\mbf{f}$ is spanned by a set of at most $k$ columns of $[\begin{array}{@{\hskip .1 cm}c@{\hskip .15 cm}c@{\hskip .1 cm}}\mbf{I}_r&\mbf{D}_r\end{array}]$.
\end{claim}
\begin{cpf}
We proceed by induction on $k$.
Consider $k=1$.
If $\mbf{f}$ has a unique nonzero entry, then $\mbf{f}$ is spanned by a column of $\mbf{I}_r$.
Otherwise, $\mbf{f}$ has a unique positive entry and a unique negative entry and is spanned by a column of $\mbf{D}_r$.
Thus, the result holds for $k = 1$.

Consider $k \ge 2$.
If $\mbf{f}$ has no positive entries (similarly, no negative entries), then $\mbf{f}$ has only negative entries (similarly, only positive entries) and is thus spanned by a set of at most $k$ unit columns.
Suppose $\mbf{f}$ has a positive entry and a negative entry.
Consider the vector $\mbf{e}_i -\mbf{e}_j$ for some $i$ such that $f_i > 0$ and some $j$ such that $f_j < 0$.
Let $\mbf{f}'$ be the vector obtained by subtracting $1$ from $f_i$ and adding $1$ to $f_j$.
By the inductive hypothesis, $\mbf{f}'$ is spanned by a set of at most $k-1$ elements, so $\mbf{f}$ is spanned by a set of at most $k$ elements.
\end{cpf}

\smallskip

We can apply Claim~\ref{claimExtensionsClique} to $\mbf{f}$ with $k \le \Delta$.
Therefore, $\mbf{f}$ is spanned by a set of most $\Delta$ columns in $[\begin{array}{@{\hskip .1 cm}c@{\hskip .15 cm}c@{\hskip .1 cm}}\mbf{I}_r&\mbf{D}_r\end{array}]$.
%
\end{proof}

We remark that the bound in Proposition~\ref{single_element_extension_of_clique} is tight for the $\Delta$-modular matrix $[\begin{array}{@{\hskip .1 cm}c@{\hskip .15 cm}c@{\hskip .15 cm}c@{\hskip .1 cm}}\mbf{I}_r & \mbf{D}_r & \mbf{f}\end{array}]$ where $\mbf{f}$ has $\Delta$ entries equal to $1$, $\Delta$ entries equal to $-1$, and all other entries equal to $0$.

\section{The Spanning Clique Case}\label{secSCC}

Our proof of Theorem \ref{main} uses a framework developed by Geelen and Nelson~\cite{GN}, which we tailor to fit $\mathcal M_{\Delta}$.
As a first step, we consider matroids $M\in \mcf{M}_{\Delta}$ that contain a complete graphic matroid as a spanning restriction.
For these matroids we show that the bound in Theorem \ref{main} holds (see Proposition~\ref{spanning clique}) and they cannot have a large number of `critical' elements (see Proposition~\ref{clique critical}); a critical element $e$ is one for which $\elem(M) - \elem(M/e)$ is large (see Proposition~\ref{clique critical} for a precise definition). 
As a second step, we show that if $M\in \mcf{M}_{\Delta}$ has more elements than the bound of Theorem \ref{main}, then $M$ has a minor that contains a complete graphic matroid as a spanning restriction and contains a large number of `critical' elements (see Section~\ref{secMainproof}).
Due to the first step this is a contradiction.

The following proposition establishes a structural property and proves Theorem \ref{main} for matroids in $\mcf{M}_{\Delta}$ with a spanning clique restriction.
A {\bf frame} for $M(K_r)$ is a basis $B$ for which each element of $M(K_r)$ is spanned by a subset of $B$ of size at most two.
Frames for $M(K_r)$ correspond to spanning stars of the graph $K_r$.

\begin{proposition} \label{spanning clique}
Let $r$ and $\Delta$ be positive integers, and let $M \in \mcf{M}_{\Delta}$ be a simple rank-$r$ matroid with a set $X \subseteq E(M)$ such that $M|X \cong M(K_{r+1})$.
Let $B \subseteq X$ be a frame for $M|X$.
Then

\smallskip
\begin{enumerate}[label=$(\arabic*)$, leftmargin = *]
\item There is a set $B' \subseteq B$ of size at most $10\Delta^2\lfloor\log_2\Delta\rfloor$ such that each non-loop element of $M/B'$ in $E(M) - X$ is parallel to an element in $B - B'$.

\smallskip
\item $|M| \le \binom{r+1}{2} + 66\Delta^7\cdot r$.
\end{enumerate}
\end{proposition}
\begin{proof}
It follows from Proposition~\ref{single_element_extension_of_clique} that each element in $E(M) - X$ is spanned by a set of at most $2\Delta$ elements of $B$.
Let $k \ge 0$ be the maximum integer for which there are pairwise disjoint subsets $Z_1, \dotsc, Z_k$ of $B$ such that for each $i \in \{1, \dotsc, k\}$:

\smallskip
\begin{enumerate}[label=(\alph*), leftmargin = *]
    \item $|Z_i| \le 2\Delta$.
    \smallskip
    \item\label{fProp2} There is an element $f_i$ in $\cl_{M/(Z_1 \cup \dots \cup Z_{i-1})}(Z_i)$ that is neither a loop nor is it parallel to an element of $\cl_{M|X}(Z_i)$ in $M/(Z_1 \cup \cdots \cup Z_{i-1})$.
\end{enumerate}

\smallskip
\noindent See Fig.~\ref{FigZ} for an example of the sets $Z_1, \dotsc, Z_k$ when a representation $\mbf{A}$ of $M$ contains $[\begin{array}{@{\hskip .1 cm}c@{\hskip .15 cm}c@{\hskip .1 cm}}\mbf{I}_r&\mbf{D}_r\end{array}]$ as a submatrix.
Note that the assumption that $[\begin{array}{@{\hskip .1 cm}c@{\hskip .15 cm}c@{\hskip .1 cm}}\mbf{I}_r&\mbf{D}_r\end{array}]$ is a submatrix cannot always be made because $\mbf{A}[X]$ may not be row equivalent to $[\begin{array}{@{\hskip .1 cm}c@{\hskip .15 cm}c@{\hskip .1 cm}}\mbf{I}_r&\mbf{D}_r\end{array}]$ using a unimodular mapping, but the figure is illustrative of the general case.

\begin{center}
    \begin{figure}[H]\label{FigZ}
    \[
    \mbf{A} =
\begin{bNiceArray}{w{c}{.25cm}w{c}{.25cm}w{c}{.25cm}w{c}{.25cm}w{c}{2.75cm}w{c}{.25cm}w{c}{.25cm}w{c}{.25cm}w{c}{.25cm}w{c}{.25cm}w{c}{.25cm}w{c}{.25cm}w{c}{.25cm}w{c}{.25cm}w{c}{.25cm}}[margin, first-row, last-row, hvlines]
   \color{black!50}Z_1&\color{black!50}Z_2&\color{black!50}\cdots&\color{black!50}Z_k&\color{black!50} B - (Z_1 \cup \cdots \cup Z_k)&&&&&&\color{black!50}f_1&\color{black!50}f_2&\color{black!50}\cdots&\color{black!50}f_k\\[.1 cm]
    \vphantom{\ddots}\mbf{I} &&&&&&\mbf{D}&&&&*&*&&*&\Block{5-1}{\cdots}\\
    \vphantom{\ddots} &\mbf{I}&&&&&&\mbf{D}&&&&*&\cdots&*&\\
     \vphantom{\ddots}& & \ddots&&&&&&\ddots&&&&\ddots&*&\\
    \vphantom{\ddots}& & & \mbf{I}&&&&&&\mbf{D}&&&&*&\\
   \vphantom{\ddots}& & & &\mbf{I}&\mbf{D}&&&&&&&&&\\
   \Block{1-10}{\color{black!50}\underbrace{\hspace{8.65 cm}}_{\text{\normalsize{$\subseteq X$}}}}&&&&&&&&&&\Block{1-4}{\color{black!50}\underbrace{\hspace{2.3 cm}}_{\text{\normalsize{$\subseteq E(M) - X$}}}}&&&&
\end{bNiceArray}
    \]
    \caption{A representation $\mbf{A}$ of $M $ when $\mbf{A}$ contains $[\begin{array}{@{\hskip .1 cm}c@{\hskip .15 cm}c@{\hskip .1 cm}}\mbf{I}_r&\mbf{D}_r\end{array}]$ as a submatrix.
    We label the notable sets in our proof, e.g., $Z_1$.
    The dimensions of each $\mbf{I}$ and $\mbf{D}$ can be inferred from the gray labels.
    The blank entries are assumed to be $\mbf{0}$, and the $*$ entries are such that $f_i$ is in the span of the columns corresponding to $Z_1 \cup \cdots \cup Z_i$ but not in the span of the columns corresponding to $Z_1 \cup \cdots \cup Z_{i-1}$.}
    \end{figure}
\end{center}

We claim that $k \le \lfloor\log_2\Delta\rfloor$.
To see why this is true, for each $i \in \{1, \dotsc, k\}$ set
\[
P_i := \cl_{M|X}(Z_i) \cup f_i.
\]
The sets $P_1, \dotsc, P_k$ are disjoint because $Z_1, \dotsc, Z_k$ are disjoint subsets of the basis $B$ and $f_i$ satisfies Property~\ref{fProp2}.
Each $P_i$ has rank at most $2\Delta$ in $M/(P_1 \cup \dots \cup P_{i-1})$ because $f_i \in \cl_{M/(Z_1 \cup \dots \cup Z_{i-1})}(Z_i)$.
Also, for each $i \in \{1,\dots, k\}$, the matroid $(M/(Z_1 \cup \dots \cup Z_{i-1}))|P_i$ is not in $\mathcal M_1$, because it has a spanning clique restriction and contains the element $f_i$ which is neither a loop nor is it parallel to an element of the spanning clique.
Thus, the union of the sets $P_1, \dotsc, P_k$ forms a $(\mcf{M}_1, 2\Delta, k)$-stack.
It then follows from Proposition~\ref{lemstacks} that $k \le \lfloor\log_2\Delta\rfloor$.

Set 
\[
B_0 := Z_1 \cup \cdots \cup Z_k.
\]
The set $B - B_0$ is the frame of the spanning clique $\cl_{M|X}(B - B_0)$ of $M/B_0$.
As $k \le \lfloor\log_2\Delta\rfloor$, we have $|B_0| \le 2\Delta \lfloor\log_2\Delta\rfloor$.
Let $\{b_1, \dots, b_d\}$ be an enumeration of $B_0$.
Note that each element in $E(M) - X$ is a either a loop of $M/B_0$ or is parallel in $M/B_0$ to an element in $X$; otherwise, we contradict the maximality of $k$.
Let $A_0$ be the set of elements of $E(M) - X$ that are loops of $M/B_0$.
For each $i \in \{1, \dotsc, d\}$, let $A_i$ be the set of elements in $E(M) - X$ that are parallel to an element in the set $\cl_{M|X}(B - B_0)$ in $M/\{b_1, \dots, b_i\}$ but not in $M/\{b_1, \dots, b_{i-1}\}$.
Note that $\{A_0, A_1, \dots, A_d\}$ is a partition of $E(M) - X$.
%
See Fig.~\ref{FigZ2} for an illustration of the sets $A_0, A_1, \dotsc, A_d$.

\begin{center}
    \begin{figure}[ht]\label{FigZ2}
    \[
    \mbf{A} =
\begin{bNiceArray}{cccw{c}{1.6cm}|cccw{c}{.5cm}|w{c}{.4cm}|w{c}{.25cm}w{c}{.25cm}w{c}{.25cm}w{c}{.25cm}|w{c}{.25cm}}[margin, first-row, hvlines, last-row]
   &&&&&&&&&\color{black!50}A_1&\color{black!50}A_2&\color{black!50}\cdots&\color{black!50}A_d&\color{black!50}A_0\\[.1 cm]
   \Block{4-4}{\mbf{I}}& & & &\Block{4-4}{}&&&&\Block{6-1}{\cdots}&*&*&\cdots&*&*\\
   &&&&&&&&&\Block{3-1}{}&*&\cdots&*&*\\
   &&&&&&&&&\Block{2-2}{}&&\ddots&\vdots&\vdots\\
   &&&&&&&&&\Block{1-3}{}&&&*&*\\
   \Block{2-4}{}&&&&\Block{2-4}{{\mbf{I}}}&&&&&\Block{2-4}{\overline{\mbf{D}}}&&&&\Block{2-1}{}\\
   &&&&&&&&&&&&\\
   \Block{1-4}{\color{black!50}\underbrace{\hphantom{B_0 = Z_1 \cup \cdots \cup Z_k}}_{\text{\normalsize$B_0 = Z_1 \cup \cdots \cup Z_k$}}}&&&&\Block{1-4}{\color{black!50}\underbrace{\hphantom{\hspace{1.8 cm}}}_{\text{\normalsize $=B-B_0$}}}&&&&&\Block{1-5}{\color{black!50}\underbrace{\hspace{3.15 cm}}_{\text{\normalsize$=E(M)-X$}}}&&&&
   \end{bNiceArray}
    \]
    \caption{This figure builds upon Fig.~\ref{FigZ}.
    %
    %
    Each column of $\overline{\mbf{D}}$ is either a standard unit vector or a difference of two unit vectors.
    We label the notable sets in our proof, e.g., $B_0$.
    The dimensions of each $\mbf{I}$ and $\overline{\mbf{D}}$ can be inferred from the gray labels.
    The blank entries are assumed to be $\mbf{0}$, and the $*$ entries denote the submatrices that can be nonzero.}
    \end{figure}
\end{center}
\smallskip

For each $i \in \{1, \dotsc, d\}$, let $A'_i \subseteq A_i$ be a maximal set such that $A_i'$ corresponds to a matching of the  clique $\cl_{M|X}(B - B_0)$.
Let $X_i$ be the corresponding set of elements of $\cl_{M|X}(B - B_0)$.
See Fig.~\ref{FigZ3} for an example.
For each $i \in \{1, \dotsc, d\}$, let $B_i \subseteq B - B_0$ be a minimal set that spans all elements of $A_i'$ in $M/B_0$.
Therefore, $|B_i| = 2|A_i'|$.
Set 
\[
B' := B_0 \cup B_1 \cup \cdots \cup B_d.
\]
Each element of $E(M) - X$ is either a loop of $M/B'$ or is parallel in $M/B'$ to an element in $B - B'$; otherwise some $A_i'$ is not maximal.
We will bound $|B'|$ by bounding $|A'_i|$ for each $i \in \{1, \dotsc, d\}$.

\begin{claim}\label{claimA'}
$|A_i'| \le 2\Delta$ for each $i \in \{1, \dotsc, d\}$.
\end{claim}
%
%
\begin{cpf}
For each $a \in A'_i$ there is some $x \in X_i$ such that $\{b_i, a, x\}$ is a circuit of $M/\{b_1, \dots, b_{i-1}\}$.
Since $X_i$ is a matching of a clique, there is a set $Y_i \subseteq X$ such that $|Y_i| = |X_i|$ and $X_i \cup Y_i$ is a circuit of $M/\{b_1, \dots, b_{i}\}$.
Then $X_i$ and $A'_i$ are both circuits of $M/\{b_1, \dots, b_{i}\}/Y_i$ because each element of $A'_i$ is parallel in $M/\{b_1, \dots, b_{i}\}$ to an element of $X_i$.
See Fig.~\ref{FigZ3} to demonstrate $X_i$ and $Y_i$ in our running example.

\begin{center}
    \begin{figure}[ht]\label{FigZ3}
    \[
    \mbf{A} =
\begin{bNiceArray}{ccw{c}{.5cm}|ccw{c}{.5cm}|w{c}{.25cm}w{c}{.25cm}w{c}{.25cm}w{c}{.25cm}|w{c}{.4cm}|w{c}{.25cm}w{c}{.25cm}w{c}{.25cm}w{c}{.25cm}w{c}{.4cm}}[margin, first-row, hvlines, last-row]
   &&&&&&&&&&&\Block{1-4}{\color{black!50}A_1}&&&&\\[.1 cm]
   \Block{3-3}{\mbf{I}} & & &\Block{3-3}{}&&&\Block{3-1}{}&\Block{3-1}{}&\Block{3-1}{}&\Block{3-1}{}&\Block{9-1}{\cdots}&\Block{1-4}{*}&&&&\Block{9-1}{\cdots}\\
   &&&&&&&&&&&\Block{2-4}{}&&&\\
   &&&&&&&&&&&&&&\\
   \Block{6-3}{}&&&\Block{6-3}{{\mbf{I}}}&&&\Block{5-1}{}\phantom{-}1&\Block{5-1}{}\phantom{-}0&\Block{5-1}{}\phantom{-}0&\Block{5-1}{}-1&&\Block{5-1}{}\phantom{-}1&\Block{5-1}{}\phantom{-}0&\Block{5-1}{}\phantom{-}0&\Block{5-1}{}\phantom{-}0\\
   &&&&&&-1&\phantom{-}0&\phantom{-}1&\phantom{-}0&&-1&\phantom{-}0&\phantom{-}1&\phantom{-}0\\
   &&&&&&\phantom{-}0&\phantom{-}1&-1&\phantom{-}0&&\phantom{-}0&\phantom{-}1&-1&\phantom{-}1\\
   &&&&&&\phantom{-}0&-1&\phantom{-}0&\phantom{-}1&&\phantom{-}0&-1&\phantom{-}0&\phantom{-}0\\
   &&&&&&\phantom{-}0&\phantom{-}0&\phantom{-}0&\phantom{-}0&&\phantom{-}0&\phantom{-}0&\phantom{-}0&-1\\
   &&&&&&&&&&&&&\\
   \Block{1-3}{\color{black!50}\underbrace{\hspace{1.5 cm}}_{\text{\normalsize$=B_0$}}}&&&\Block{1-3}{\color{black!50}\underbrace{\hphantom{\hspace{1.5 cm}}}_{\text{\normalsize $=B-B_0$}}}&&&\Block{1-2}{\color{black!50}\underbrace{\hspace{1.1 cm}}_{\text{\normalsize$=X_1$}}}&&\Block{1-2}{\color{black!50}\underbrace{\hspace{1.1 cm}}_{\text{\normalsize$=Y_1$}}}&&&\Block{1-2}{\color{black!50}\underbrace{\hspace{1.1 cm}}_{\text{\normalsize$=A_1'$}}}&&
   \end{bNiceArray}
    \]
    \caption{This figure builds upon Figure~\ref{FigZ2}.
    The set $A_1'$ corresponds to elements that form a matching in $M/B_0$. 
    There is a corresponding set of elements $X_1 \subseteq \cl_{M|X}(B-B_0)$; in this example, the term `corresponding' means that the columns of $X_1$ and $A_1'$ are equal on the components supported by $B-B_0$.
    Finally, there exists a set $Y_1 \subseteq X$ such that $X_1 \cup Y_1$ forms a circuit in $M/B_0$; in this example, $M/B_0$ is the matrix formed by removing the first $|B_0|$ rows from $\mbf{A}$; one sees that the resulting columns corresponding to $X_1 \cup Y_1$ form a circuit.}
    \end{figure}
\end{center}
\smallskip

Set
\[
S := \big(M/\{b_1, \dots, b_{i-1}\}/Y_i\big)|\big(A'_i \cup X_i \cup b_i\big).
\]
We claim that $S$ is a spike with tip $b_i$.
The simplification of $S/b_i$ is a circuit and each parallel class of $S/b_i$ has size two, consisting of one element from $A_i'$ and one element from $X_i$.
Thus, it suffices to show that $S$ is simple.
Suppose first that $S$ has a loop $\ell$.
The contraction $S/b_i$ has no loops, so $\ell$ must be $b_i$.
However, $Y_i$ is independent in $M/\{b_1, \dots, b_{i}\}$, so it does not span $b_i$ in $S$; this contradicts that $b_i$ is a loop in $S$.
Suppose next that $S$ has a parallel pair $\{\ell_1, \ell_2\}$. 
Again using the fact that the contraction $S/b_i$ has no loops, $\{\ell_1, \ell_2\}$ must be equal to a parallel pair $\{a,x\}$ of $S/b_i$.
However, $b_i$ is a non-loop, so $a$ and $x$ are loops of $S/b_i$, which is a contradiction.
Therefore $S$ is a spike.
According to Proposition~\ref{lemspikes}, it holds that $|A_i'| \le 2\Delta$.
\end{cpf}

\smallskip
Claim~\ref{claimA'} implies that $|B_i| \le 2|A_i'| \le 4\Delta$ for each $i \in \{1, \dotsc, d\}$.
Hence,
\[
|B'| \le 2\Delta \lfloor\log_2 \Delta\rfloor + 4\Delta (2\Delta \lfloor\log_2 \Delta\rfloor) \le 10\Delta^2 \lfloor\log_2 \Delta\rfloor.
\]
For each $b \in B - B'$, there are at most $\Delta^2\binom{|B'| + 2}{2}$ elements in $E(M) - X$ spanned by $B' \cup b$ by Theorem \ref{poly_Delta_bound}.
Thus, 
\[
|M| = |X| + |E(M) - X|\le \binom{r + 1}{2} + \Delta^2\binom{|B'| + 2}{2}\cdot r.
\]
It follows that
\begin{align*}
|M|  \le & \binom{r+1}{2} + (50 \Delta^6\lfloor\log_2\Delta\rfloor^2 + 15\Delta^2\lfloor\log_2\Delta\rfloor + 1) \cdot r\\
\le & \binom{r+1}{2} + 66\Delta^7 \cdot r.
\end{align*}
\end{proof}

Let $f$ be an element of a matroid $M \in \mcf 
M_{\Delta}$, and let $Z$ be the union of all long lines of $M$ through $f$.
We say that $f$ is {\bf critical} if the number of long lines of $M$ through $f$ is greater than $r(Z) + 70\Delta^7$.
We refer to a transversal of $$\{L - e \colon \textrm{$L$ is a long line through $e$}\}$$
as a \textbf{transversal} of the long lines through $e$. 
We next use Proposition \ref{spanning clique} to show that a matroid in $\mathcal M_{\Delta}$ with a spanning clique restriction has a bounded number of critical points.

\begin{proposition} \label{clique critical}
Let $r$ and $\Delta$ be positive integers, and let $M \in \mcf{M}_{\Delta}$ be a simple rank-$r$ matroid with an $M(K_{r+1})$-restriction.
Then $M$ has at most $200 \Delta^7$ critical points.
\end{proposition}
\begin{proof}
Let $X \subseteq E(M)$ such that $M|X \cong M(K_{r+1})$, and let $B \subseteq X$ be a frame for $M|X$.
By Proposition \ref{spanning clique} there is a set $B' \subseteq B$ of size at most $10\Delta^2\lfloor\log_2\Delta\rfloor$ such that each non-loop element of $M/B'$ in $E(M) - X$ is parallel to an element in $B - B'$.

We first bound the number of critical points in $X$.
Assume to the contrary that there are more than $\binom{|B'|+1}{2}$ critical points in $X$.
Since $|\cl_{M|X}(B')| = \binom{|B'|+1}{2}$, there is a critical point $f$ that is a non-loop of $M/B'$.
%
Since the simplification of $(M|X)/B'$ is a clique with frame $B - B'$, either $f$ is parallel in $M/B'$ to some $b \in B - B'$, or is in a $3$-element circuit in $M/B'$ with some $b,b' \in B - B'$.
We only consider when $\{f, b, b'\}$ is a circuit as the former is proved similarly.

\begin{center}
    \begin{figure}[ht]
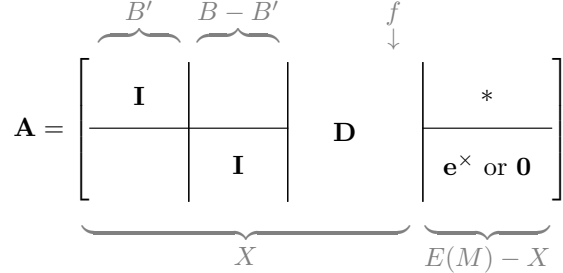
\label{FigX}
    \[
    \mbf{A} =
\begin{bNiceArray}{w{c}{.25cm}w{c}{.25cm}|w{c}{.25cm}w{c}{.25cm}|cccw{c}{.25cm}|cccw{c}{.25cm}}[margin, first-row, last-row, hvlines]
   \Block{1-2}{\color{black!50}\overbrace{\hphantom{~~~B'~~~}}^{\text{\normalsize $B'$}}}&&\Block{1-2}{\color{black!50}\overbrace{\hphantom{~~~B'~~~}}^{\text{\normalsize $B-B'$}}}&&&&&\color{black!50}\overset{\text{\normalsize $f$}}{\text{\small$\downarrow$}}&&&\\[.1 cm]
   \Block{2-2}{\mbf{I}}&& \Block{2-2}{}&&\Block{4-4}{\mbf{D}}&&&&\Block{2-4}{*}&&&\\
   &&&&&&&&&&&\\
   \Block{2-2}{}&&\Block{2-2}{\mbf{I}}&&&&&&\Block{2-4}{\mbf{e}^{\times}~\text{or}~\mbf{0}}&&&\\
   &&&&&&&&&&&\\[.1 cm]
   \Block{1-8}{\color{black!50}\underbrace{\hphantom{\hspace{4.35 cm}}}_{\text{\normalsize$X$}}}&&&&&&&&\Block{1-4}{\color{black!50}\underbrace{\hphantom{\hspace{1.65 cm}}}_{\text{\normalsize{$E(M)-X$}}}}
   \end{bNiceArray}
    \]
    \caption{ We label the notable sets in our proof, e.g., $B'$, using gray font to indicate the corresponding set of columns.
    A possible location of the column corresponding to $f$ is also drawn.
    The dimensions of each $\mbf{I}$ and $\mbf{D}$ can be inferred from the gray labels.
    The blank entries are assumed to be $\mbf{0}$, and the $*$ entries are arbitrary.
    In the bottom right block marked as `$\mbf{e}^{\times}~\text{or}~\mbf{0}$', each column is either a standard unit column or $\mbf{0}$.}
    \end{figure}
\end{center}

Let $Z$ be the union of all long lines of $M$ through $f$.
Each line in $Z$ is contained in $\cl_M(B' \cup \{b, b'\})$ or contains at least two elements in $X$.
If this is not the case, then there is long line $L$ in $Z$ that is not contained in $\cl_M(B' \cup \{b,b'\})$ and does not contain two elements in $X$.
The line $L$ contains some non-loop element $g$ that is parallel in $M/B'$ to some $b_1 \notin \{b,b'\}$.
However, since $f$ and $g$ are in a common $3$-element circuit of $M$ and neither is a loop of $M/B'$, they are parallel in $M/B'$, which a contradiction.

Next, we count the number of lines in $Z$ that are in $\cl_M(B' \cup \{b, b'\})$, and the number of lines in $Z$ that contain at least two elements in $X$. 
For $\cl_M(B' \cup \{b, b'\})$, we note that the lines in $Z$ correspond to points of $(M|\cl_M(B' \cup \{b,b'\}))/f$.
Since $(M|\cl_M(B' \cup \{b,b'\}))/f$ has rank $|B'| + 1$ and is in $\mathcal M_{\Delta}$, it has at most $\binom{|B'|+2}{2}\Delta^2$ points by Theorem \ref{poly_Delta_bound}.
Therefore at most $\binom{|B'|+2}{2}\Delta^2$ lines in $Z$ through $f$ are contained in $\cl_M(B' \cup \{b, b'\})$.
%
%
For lines in $Z$ that contain at least two elements in $X$, note that any transversal of long lines of $M|(X \cup f)$ through $f$ is independent in $M/f$ because spikes are not graphic \cite[pg. 63]{GN}.
Therefore, there are at most $r(Z)$ many lines in $Z$ that contain at least two elements in $X$.
In total, we have that there are at most $r(Z) + \binom{|B'|+2}{2}\Delta^2$ long lines of $M$ through $f$.
Since $\binom{|B'|+2}{2}\Delta^2 \le 70 \Delta^7$ this means that  $f$ is not critical, which is a contradiction.
Therefore, $M$ has at most $\binom{|B'|+1}{2}$ critical points in $X$.

%

We now bound the number of critical points in $E(M) - X$.
Assume to the contrary that there are more than $\binom{|B'|+1}{2}\Delta^2$ critical points in $E(M) - X$.
By Theorem \ref{poly_Delta_bound} we have $|\cl_M(B')| \le \binom{|B'|+1}{2}\Delta^2$, and thus there is a critical point $f$ that is a non-loop of $M/B'$.
The definition of $B'$ then implies that $f$ is parallel in $M/B'$ to some $b \in B - B'$.
There are at most $\binom{|B'|+2}{2}\Delta^2$ lines that contain an element in $\cl_M(B' \cup b)$.
We also have the following:

\begin{claim}
There are at most three long lines of $M$ through $f$ that each contain two elements of $M|X$.
\end{claim}
\begin{cpf}
Let $\mathcal L$ be the set of lines of $M|X$ that span $f$ in $M$.
Each set in $\mathcal L$ corresponds to a triangle or a two-edge matching of $K_{r(M)+1}$, and for all distinct $L_1, L_2 \in \mathcal L$ we have $L_1 \cap L_2 =\varnothing$  and $r(L_1 \cup L_2) = 3$.
These properties imply that if there is a $3$-element set in $\mathcal L$ then $|\mathcal L| = 1$, and that the union of any two $2$-element sets in $\mathcal L$ corresponds to a $4$-cycle of $K_{r(M)+1}$.
Therefore $|\mathcal L| \le 3$.
\end{cpf}


%

\smallskip
Let $F$ be the set of long lines of $M$ through $f$ that are neither in $\cl_M(B' \cup b)$ nor do they contain two elements of $M|X$.
We will show that any transversal of the lines in $F$ is independent in $M/f$.
Let $L$ be a line in $F$.
There exists a point $\ell \in L$ that is parallel in $M/B'$ to a point $b_1 \ne b$.
The points $f$ and $\ell$ are independent in $M/B'$, so $(M/B')|L = M|L$.
The line $L$ contains the unique element $g$ of $M|X$ spanned by $\{b,b_1\}$.
Hence, for each line in $F$ there is a point $g_L \in X$ spanned in $M$ by $\{b,b_L\}$ for some $b_L \in B - (B' \cup b)$.
Note that $g_L$ and $b_L$ are parallel in $M/b$, and so the set $\{g_L:\ L \in F\} \subseteq X$ is independent in $M/(B' \cup b)$.
Since $\{g_L:\ L \in F\}$ is independent in $M/(B' \cup b)$ and $f \in \cl_M(B' \cup b)$, it follows that $\{g_L:\ L \in F\}$ is also independent in $M/f$.
Therefore, there is an independent transversal of all but $3 + \binom{|B'|+2}{2}\Delta^2$ long lines of $M$ through $f$.
Since $3 + \binom{|B'|+2}{2}\Delta^2 \le 70\Delta^7$, this implies that $f$ is not critical.
However, this is a contradiction.

Since $M$ has at most $\binom{|B'|+1}{2}$ critical points in $X$ and at most $\binom{|B'|+1}{2}\Delta^2$ critical points in $E(M) - X$, it has at most $\binom{|B'|+1}{2} + \binom{|B'|+1}{2}\Delta^2 \le 200 \Delta^7$ critical points.
\end{proof}

\section{The Proof of Theorem \ref{main}}\label{secMainproof}

We first state a sequence of known results that will allow us to reduce Theorem~\ref{main} to the spanning clique case.
We say that $M$ is {\bf vertically $s$-connected} if there is no partition $(X,Y)$ of $E(M)$ for which $r_M(X) + r_M(Y) - r(M) < j$ and $\min(r_M(X), r_M(Y)) \ge j$, where $j < s$.
Rather than provide an example of a vertically $s$-connected matrix, it is simpler to provide an example of a matrix that {\it is not} vertically $s$-connected; see Fig.~\ref{Figvsc}.

\begin{center}
    \begin{figure}[H]\label{Figvsc}
    \[
\mbf{A} = \begin{bNiceArray}{cccccc}[margin, hvlines, last-col]
   \Block{3-3}{\mbf{B}}&&&\Block{2-3}{\mbf{0}}&&&\\
   &&&&&&\\
   \rowcolor{black!7}&&&\Block{3-3}{\mbf{C}}&&&\color{black!50} \} < s - 1\\
   \Block{2-3}{\mbf{0}}&&&&&&\\
   &&&&&&
   \end{bNiceArray}
    \]
    \caption{Suppose the matrices $\mbf{B}$ and $\mbf{C}$ satisfy $\rank(\mbf{B}) \ge s - 1$ and $\rank(\mbf{C}) \ge s - 1$.
    There are fewer than $s-1$ rows of $\mbf{A}$ that correspond to rows in both $\mbf{B}$ and $\mbf{C}$.
    Hence, the sets of columns $X$ and $Y$ corresponding to $\mbf{B}$ and $\mbf{C}$, respectively, demonstrate that $\mbf{A}$ is not vertically $s$-connected.}
    \end{figure}
\end{center}

We will apply the following result from~\cite{GNW2021} with $M$ as the vector matroid of a $\Delta$-modular matrix and $p(x) = \binom{x + 1}{2} + (70 \Delta^7 + 8 \Delta^3 \lfloor\log_2 \Delta\rfloor^2)\cdot x$.
We will also take $\ell = 2\Delta$, because $\Delta$-modular matroids have no $U_{2, 2\Delta + 2}$-minor by Proposition 8.6.1 in \cite{GNW2021}.

\begin{theorem}[Geelen, Nelson, Walsh \cite{GNW2021}] \label{reduction}
There is a function $f_{\ref{reduction}}\colon \mbb{R}^6\to \mbb{Z}$ such that the following holds for all integers $r,s,\ell$ with $r,s\ge 1$ and $\ell \ge 2$ and any real polynomial $p(x)=ax^2+bx+c$ with $a>0$:

\smallskip%
\noindent If $M$ is a matroid with no $U_{2, \ell+2}$-minor and satisfies $r(M)\ge f_{\ref{reduction}}(a,b,c,\ell,r,s)$ and $\elem(M)>p(r(M))$, then $M$ has a minor $N$ with $\elem(N)>p(r(N))$ and $r(N)\ge r$ such that either
\begin{enumerate}[label=$(\arabic*)$, leftmargin = *]
\item\label{GNW1} $N$ has a spanning clique restriction, or
\item\label{GNW2} $N$ is vertically $s$-connected and has an $s$-element independent set $S$ such that $\elem(N)-\elem(N/e)>p(r(N))-p(r(N)-1)$ for each $e\in S$.
\end{enumerate}
\end{theorem}

Theorem~\ref{reduction} says that if there is a large counterexample $M$ to Theorem \ref{main}, then there exists a counterexample $N$ with either a spanning clique restriction or structure given by outcome~\ref{GNW2}, which will lead to a large independent set of critical points.
Outcomes~\ref{GNW1} and~\ref{GNW2} are handled by Proposition~\ref{spanning clique} and Proposition~\ref{clique critical}, respectively.
To make use of Proposition~\ref{clique critical}, we must take the critical points given by outcome~\ref{GNW2} and contract them into the span of a clique minor of $N$.
We do this with the following theorem from~\cite{GN}, which says that any bounded-rank set in a matroid with large vertical connectivity can be contracted into the span of a clique minor.

\begin{theorem}[Geelen and Nelson \cite{GN}] \label{connectivity}
There is a function $f_{\ref{connectivity}}\colon \mbb{Z}^2 \to \mbb{Z}$ such that the following holds for all integers $s,m,\ell$ with $m > s > 1$ and $\ell \ge 2$:

\smallskip

\noindent If $M$ is a vertically $s$-connected matroid with an $M(K_{f_{\ref{connectivity}}(m, \ell) + 1})$-minor and no $U_{2, \ell + 2}$-minor, and $X \subseteq E(M)$ satisfies $r_M(X) < s$, then $M$ has a rank-$m$ minor $N$ with an $M(K_{m+1})$-restriction such that $X \subseteq E(N)$ and $N|X = M|X$.
\end{theorem}

In order to apply Theorem \ref{connectivity}, we must guarantee the existence of a large clique minor.
The existence of such a minor is given by the following theorem from \cite{GW}, which says if a matroid has size greater than a linear function of its rank, then the matroid has a large clique minor.

\begin{theorem}[Geelen and Whittle \cite{GW}] \label{clique minor}
There is a function $f_{\ref{clique minor}} \colon \mbb{Z}^2 \to \mbb{Z}$ such that the following holds for all integers $n, \ell \ge 2$:

\smallskip
\noindent If $M$ is a matroid with no $U_{2, \ell+2}$-minor, and $\elem(M) > f_{\ref{clique minor}}(n, \ell) \cdot r(M)$, then $M$ has an $M(K_{n+1})$-minor.
\end{theorem}

We need a new lemma before proving Theorem \ref{main}.
The lemma will help find critical points from outcome (2) of Theorem \ref{reduction}.
For a family $\mathcal F$ of sets, we write $\cup \mathcal F$ for $\cup_{F \in \mathcal F}F$.

\begin{lemma} \label{lm:lines}
Let $M$ be a simple matroid, and let $\Delta$ be a positive integer.
Let $S$ be an independent set for which each element is on at least $4 \Delta^2 \lfloor\log_2\Delta\rfloor^2 + 1$ lines of $M$, each with at least four points. 
If $M$ is $\Delta$-modular, then $|S| \le 2 \lfloor\log_2 \Delta\rfloor$.
\end{lemma}
\proof
Let $\mcf{L}$ be a maximal collection of lines of $M$ such that each has least four points and $r_M(\cup \mcf{L}) = 2|\mcf{L}|$.
Since $r_M(\cup \mcf{L}) = 2|\mcf{L}|$, the matroid $M|\cup \mathcal L$ is the direct sum of the lines in $\mathcal L$, and so the sets in $\mcf{L}$ form an $(\mcf{M}_1, 2, |\mcf{L}|)$-stack.
Since $M$ is $\Delta$-modular, Proposition~\ref{lemstacks} implies that $|\mcf{L}| \le \lfloor\log_2 \Delta\rfloor$.
Hence, $r_M(\cup \mcf{L}) \le 2\lfloor\log_2 \Delta\rfloor$.

Suppose that $|S| > 2\lfloor\log_2 \Delta\rfloor$.
There is some $f \in S$ that is not spanned by $\cup \mcf{L}$.
By hypothesis, there are at least $4 \Delta^2 \lfloor\log_2\Delta\rfloor^2 + 1$ lines of $M$ through $f$ with at least four points.
So $M/f$ has at least $4 \Delta^2 \lfloor\log_2\Delta\rfloor^2 + 1$ parallel classes with at least three points; let $\mcf{X}_f$ be this set of parallel classes.
Since $4 \Delta^2 \lfloor\log_2\Delta\rfloor^2 + 1 > \Delta^2\binom{2 \lfloor\log_2 \Delta\rfloor + 1}{2}$ and $(M/f)|\cup\mcf{X}_f \in \mcf{M}_{\Delta}$, Theorem \ref{poly_Delta_bound} implies that $r_{M/f}(\cup \mcf{X}_f) \ge 2 \lfloor\log_2 \Delta\rfloor + 2$. 
If each set in $\mcf{X}_f$ has rank less than two in $M/\cup \mcf{L}$, then $r_{M/\cup \mcf{L}}(\cup \mcf{X}_f) = 1$, since each element is a loop or parallel to $f$.
But then $r_M(\cup \mcf{X}_f) \le 2|\mcf{L}| + 1 < 2 \lfloor\log_2 \Delta\rfloor + 2$, which contradicts that $r_{M/f}(\cup \mcf{X}_f) \ge 2\lfloor\log_2\Delta\rfloor + 2$.
Thus, there is a set $P \in \mcf{X}_f$ that has rank two in $M/\cup \mcf{L}$.
Then $L = P \cup f$ has rank two in $M/\cup \mcf{L}$ and contains at least four points of $M/\cup \mcf{L}$.
However, this implies that $\mcf{L} \cup \{L\}$ contradicts the maximality of $\mcf{L}$.
\endproof

\smallskip

The following result implies Theorem \ref{main}.

\begin{theorem} \label{final main}
There is a function $f_{\ref{final main}} \colon \mbb{Z} \to \mbb{Z}$ such that the following holds for each positive integer $\Delta$ and each integer $r \ge f_{\ref{final main}}(\Delta)$:

\smallskip
\noindent If $\mbf{A}$ is a rank-$r$ $\Delta$-modular matrix, then $\mbf{A}$ has at most $\binom{r + 1}{2} + (70 \Delta^7 + 8 \Delta^3 \lfloor\log_2 \Delta\rfloor^2) \cdot r$ nonzero, pairwise non-parallel columns.
\end{theorem}
\proof
Define $p:\mbb{R}\to \mbb{R}$ to be
%
\[
p(x) := \binom{x + 1}{2} + \left(70 \Delta^7 + 8 \Delta^3 \lfloor\log_2 \Delta\rfloor^2\right) \cdot x.
\]
Set $\ell_1 := 2\Delta$ and $s_1 := 200^2 \Delta^{15} + 1$.
Let $r_1$ be an integer such that $p(x) > f_{\ref{clique minor}}(f_{\ref{connectivity}}(s_1 + 1, \ell_1), \ell_1) \cdot x$ for all $x \ge r_1$.
Define
\[
f_{\ref{final main}}(\Delta) := f_{\ref{reduction}}\left(\frac{1}{2},\ \frac{1}{2} + 70 \Delta^7 + 8 \Delta^3 \lfloor\log_2 \Delta\rfloor^2,\ 0,\ \ell_1,\ r_1,\ s_1\right).
\]

Suppose $\mbf{A}$ is a rank-$r$ $\Delta$-modular matrix with more than $\binom{r + 1}{2} + (70 \Delta^7 + 8 \Delta^3 \lfloor\log_2 \Delta\rfloor^2) \cdot r$ nonzero, pairwise non-parallel columns, where $r \ge f_{\ref{final main}}(\Delta)$.
The vector matroid $M$ of $\mbf{A}$ has no $U_{2, \ell_1 + 2}$-minor by Proposition 8.6.1 in \cite{GNW2021}.
By Theorem \ref{reduction} with 
\[
(a, b, c, \ell, r, s) = \left(\frac{1}{2},\ \frac{1}{2} + 70 \Delta^7 + 8 \Delta^3 \lfloor\log_2 \Delta\rfloor^2,\ 0,\ \ell_1,\ r_1,\ s_1\right),
\]
there is a simple minor $N$ of $M$ such that $\elem(N) > p(r(N))$ and $r(N) \ge r_1$ and either

\smallskip
\begin{enumerate}[label=(\alph*), leftmargin = *]
\item\label{sprc1} $N$ has a spanning clique restriction, or
\item\label{sprc2} $N$ is vertically $s_1$-connected and has an $s_1$-element independent set $S$ such that $\elem(N)-\elem(N/f)>p(r(N))-p(r(N)-1)$ for each $f\in S$.
\end{enumerate}
\smallskip

The matroid $N$ is in $\mcf{M}_{\Delta}$ because $\mcf{M}_{\Delta}$ is minor-closed \cite[Prop. 8.6.1]{GNW2021}.
Outcome~\ref{sprc1} does not hold due to Proposition \ref{spanning clique}.
Hence, Outcome~\ref{sprc2} holds, and each $f \in S$ satisfies 
\[
\elem(N) - \elem(N/f) > r(N) + 70 \Delta^7 + 8 \Delta^3 \lfloor\log_2 \Delta\rfloor^2.
\]
By Lemma \ref{lm:lines}, at most $2 \lfloor\log_2 \Delta\rfloor$ elements in $S$ are on at least $4 \Delta^2 \lfloor\log_2 \Delta\rfloor^2 + 1$ lines with at least four points.
Therefore, there is a set $S_1 \subseteq S$ such that $|S_1| = 200 \Delta^7 \ge |S| - 2 \lfloor\log_2 \Delta\rfloor$, and each $f \in S_1$ is on at most $4 \Delta^2 \lfloor\log_2 \Delta\rfloor^2$ lines with at least four points.

\begin{claim} \label{critical}
For each $f \in S_1$, there is a set $\mcf{X}_f$ of long lines of $N$ through $f$ such that $|\mcf{X}_f| \ge r_N(\cup \mcf{X}_f) + 70 \Delta^7$ and $r_N(\cup \mcf{X}_f) \le 200 \Delta^8$.
In particular, $f$ is a critical element of $N$.
\end{claim}

\begin{cpf}
We will use the fact that each $f \in N$ satisfies 
\[
\elem(N) - \elem(N/f) = 1 + \sum_{L \in \mcf{L}_N(f)} |L| - 2,
\]
where $\mcf{L}_N(f)$ is the set of long lines of $N$ that contain $f$.
Each $f \in S$ satisfies 
\begin{equation}\label{eqBoundLines}
r(N) + 70 \Delta^7 + 8 \Delta^3 \lfloor\log_2 \Delta\rfloor^2 \le \sum_{L \in \mcf{L}_N(f)} |L| - 2.
\end{equation}

Let $f \in S_1$.
The element $f$ is on at most $4 \Delta^2 \lfloor\log_2 \Delta\rfloor^2$ lines with at least four points, and each line of $N$ has at most $2 \Delta + 1$ points.
Therefore, the lines through $f$ with at least four points contribute at most $ (2 \Delta)(4 \Delta^2 \lfloor\log_2 \Delta\rfloor^2)$ to the right-hand side of~\eqref{eqBoundLines}.
%
Thus, the lines through $f$ with exactly three points contribute at least $r(N) + 70 \Delta^7$ to the right-hand side of~\eqref{eqBoundLines}. 
This implies that $f$ is on at least $r(N) + 70 \Delta^7$ long lines of $N$.

Let $\mcf{X}_f$ be a set of long lines of $N$ through $f$ such that $|\mcf{X}_f| \ge r_N(\cup \mcf{X}_f) + 70 \Delta^7$, and $|\mcf{X}_f|$ is minimal. 
Note that $\mcf{X}_f$ is well-defined because the set of all long lines of $N$ through $f$ is an option for $\mcf{X}_f$.
We have $|\mcf{X}_f| = r_N(\cup \mcf{X}_f) + 70 \Delta^7$, so each transversal of $\mcf{X}_f$ is a union of at most $70 \Delta^7$ circuits of $N/f$. 
Each circuit of $N/f$ contained in a transversal of $\mcf{X}_f$ has size at most $2\Delta$ by Proposition~\ref{lemspikes}.
Hence, $|\mcf{X}_f| \le 140 \Delta^8$, which implies that $r_N(\cup \mcf{X}_f) \le 140 \Delta^8 \le 200 \Delta^8$.
\end{cpf}

\smallskip

By Theorem \ref{clique minor} and the definition of $r_1$, $M(K_{f_{\ref{connectivity}}(s_1 + 1, \ell_1) + 1})$ is a minor of $N$.
Set $X = S_1 \cup (\cup_{f \in S_1}(\cup \mcf{X}_f))$.
Claim~\ref{critical} implies that 
\[
r_N(X) \le |S_1| \cdot 200 \Delta^8 = 200^2\Delta^{15} < s_1.
\]
Applying Theorem \ref{connectivity} with $(s, m, \ell) = (s_1, s_1 + 1, \ell_1)$, the matroid $N$ has a minor $N_1$ with a spanning clique restriction such that $X \subseteq E(N_1)$ and $N_1|X = N|X$.
The set $S_1$ is an independent set of $200\Delta^7$ critical elements of $N_1$ because $N_1|X = N|X$.
However, $N_1$ has a spanning clique restriction, which contradicts Proposition \ref{clique critical}.
\endproof

\section{Future Work}\label{secFutureWork}
A clear direction for future work is to resolve Problem~\ref{extremal} for all $\Delta$ and $r$.
A lower bound of $\binom{r+1}{2} + (\Delta-1)(r-1)$ is given for all $\Delta$ and $r$ by the matrix $[\begin{array}{@{\hskip .1 cm}c@{\hskip .15 cm}c@{\hskip .15 cm}c@{\hskip .1 cm}}\mbf{I}_r &\mbf{D}_r & \mbf{X}_r\end{array}]$, where $\mbf{X}_r$ has columns $k\mbf{e}_1 - \mbf{e}_j$ for all $k \in \{2,3,\dotsc, \Delta\}$ and $j \in \{2,3,\dotsc, r\}$.
This lower bound is the answer to Problem~\ref{extremal} when $\Delta \le 2$ and $r \ge 6$, and it was conjectured in \cite{LPSX2021} that it is the answer for all $\Delta$ and large enough $r$.
However, this conjecture was disproved for $\Delta \in \{4,8,16\}$ by Averkov and Schymura~\cite{AS2022}.
A qualitative version of the conjecture in~\cite{LPSX2021} is a bound in $\binom{r+1}{2} + O((\Delta-1)(r-1))$; such a bound is still possible based on Averkov and Schymura's results.

Problem~\ref{extremal} is even open for the rank-$2$ case and most values of $\Delta$.
%
%
The following question was also posed in \cite{OW2021}.

\begin{problem} \label{rank-2}
What is the maximum number of nonzero, pairwise non-parallel columns of a rank-$2$ $\Delta$-modular matrix?
\end{problem}

It is not difficult to show that the answer to Problem~\ref{rank-2} is at least $\Delta + 2$ and at most $\sfrac{3}{2}\cdot \Delta+1$.
Another upper bound is $p+1$, where $p$ is the smallest prime greater than $\Delta$.
The bound $p+1$ holds because every $\Delta$-modular matroid is representable over the field $\GF(p)$ and therefore has no $U_{2,p+2}$-minor, as observed by Geelen, Nelson, and Walsh~\cite{GNW2021}.
Problem \ref{rank-2} can be solved computationally for small values of $\Delta$, so this should not pose an issue for answering Problem \ref{extremal} for small values of $\Delta$. 
In particular, we hope that our techniques can be refined to solve Problem \ref{extremal} in the case that $\Delta = 3$.

\section*{Acknowledgements}
J. Paat was supported by a Natural Sciences and Engineering Research Council of Canada (NSERC) Discovery Grant [RGPIN-2021-02475].

\bibliographystyle{plain}
\bibliography{references.bib}

\end{document}